\documentclass[letterpaper,11pt,twoside,keywordsasfootnote,addressatend,noinfoline]{article}
\usepackage{fullpage}
\usepackage[english]{babel}
\usepackage{amssymb}
\usepackage{amsmath}
\usepackage{theorem}
\usepackage{epsfig}
\usepackage{subfigure}
\usepackage{imsart}

\theorembodyfont{\slshape}
\newtheorem{theor}{Theorem}

\newtheorem{lemma}[theor]{Lemma}

\newenvironment{proof}{\noindent{\scshape Proof.}}{\hspace{2mm} $\square$}

\newcommand{\N}{\mathbb{N}}
\newcommand{\Z}{\mathbb{Z}}
\newcommand{\R}{\mathbb{R}}
\newcommand{\ep}{\epsilon}
\newcommand{\ind}{\mathbf 1}

\DeclareMathOperator{\card}{card \,}
\DeclareMathOperator{\sign}{sign \,}

\begin{document}

\begin{frontmatter}

\title     {Galam's bottom-up hierarchical system and \\ public debate model revisited}
\runtitle  {Galam's bottom-up hierarchical system and public debate model revisited}
\author    {N. Lanchier\thanks{Research supported in part by NSF Grant DMS-10-05282.} and N. Taylor}
\runauthor {N. Lanchier and N. Taylor}
\address   {School of Mathematical and Statistical Sciences, \\ Arizona State University, \\ Tempe, AZ 85287, USA.}

\begin{abstract} \ \
 This article is concerned with the bottom-up hierarchical system and public debate model proposed by Galam, as well as a spatial
 version of the public debate model.
 In all three models, there is a population of individuals who are characterized by one of two competing opinions, say opinion~$-1$
 and opinion~$+1$.
 This population is further divided into groups of common size~$s$.
 In the bottom-up hierarchical system, each group elects a representative candidate, whereas in the other two models, all the
 members of each group discuss at random times until they reach a consensus.
 At each election/discussion, the winning opinion is chosen according to Galam's majority rule:
 the opinion with the majority of representants wins when there is a strict majority while one opinion, say opinion~$-1$, is chosen
 by default in case of a tie.
 For the public debate models, we also consider the following natural updating rule that we shall call proportional rule:
 the winning opinion is chosen at random with a probability equal to the fraction of its supporters in the group.
 The three models differ in term of their population structure:
 in the bottom-up hierarchical system, individuals are located on a finite regular tree, in the non-spatial public debate model, they
 are located on a complete graph, and in the spatial public debate model, they are located on the~$d$-dimensional regular lattice.
 For the bottom-up hierarchical system and non-spatial public debate model, Galam studied the probability that a given opinion wins
 under the majority rule and assuming that individuals' opinions are initially independent, making the initial number of supporters
 of a given opinion a binomial random variable.
 The first objective of this paper is to revisit his result assuming that the initial number of individuals in favor of a given
 opinion is a fixed deterministic number.
 Our analysis reveals phase transitions that are sharper under our assumption than under Galam's assumption, particularly
 with small population size.
 The second objective is to determine whether both opinions can coexist at equilibrium for the spatial public debate model under the
 proportional rule, which depends on the spatial dimension.
\end{abstract}



\end{frontmatter}


\section{Introduction}
\label{sec:intro}

\indent Galam's bottom-up hierarchical system and public debate model \cite{galam_2008} are used to understand voting behaviors of
 two competing opinions in democratic societies.
 In his models, Galam assumes that initially individuals in the population are independently in favor of one opinion with a fixed
 probability, making the initial number of that type of opinion a binomial random variable.
 This analysis revisits Galam's models by assuming that the initial number of individuals in favor of an opinion is a fixed
 deterministic number, which is more realistic when analyzing small populations.
 This paper is also concerned with a spatial version of Galam's public debate model introduced in \cite{lanchier_neufer_2013}.
 Before stating our results, we start with a detailed description of these three models. \vspace*{5pt}


\noindent {\bf Bottom-up hierarchical system} --
 The bottom-up hierarchical system \cite{galam_2008} is a stochastic process that depends on two parameters:
 the group size~$s$ and the number of voting steps~$N$, which are both positive integers.
 The structure of this model, which is displayed in Figure \ref{fig:voting-model}, begins with a population of~$s^N$ individuals in
 favor of either opinion~$+1$ or opinion~$-1$ on the bottom level.
 This population is further divided into groups of size $s$ and local majority rules determine a representative candidate of each
 group who then ascends to another group at the next lowest level.
 This process continues until a single winner at level 0 is elected.
 When the group size $s$ is odd, majority rule is well defined, whereas when the group size $s$ is even, a bias is introduced
 favoring a predetermined type, say opinion~$-1$, if there is a tie.
 That is, the representative candidates are determined at each step according to the majority rule whenever there is a strict
 majority but is chosen to be the one in favor of opinion~$-1$ in case of a tie.
 This assumption is justified by Galam \cite{galam_2008} based on the principle of social inertia.
 More formally, one can think of the model as a rooted regular tree with degree $s$ and $N$ levels plus the root.
 Denote by
 $$ X_n (i) \quad \hbox{for} \quad n = 0, 1, \ldots, N \quad \hbox{and} \quad i = 1, 2, \ldots, s^n $$
 the opinion of the $i$th node/individual at level~$n$.
 Then, the opinion of each node is determined from the configuration of opinions $X_N$ at the bottom level and the recursive rule:
 $$ \begin{array}{l} X_n (i) \ := \ \sign (\sum_{j = 1}^s X_{n + 1} (s (i - 1) + j) - 1/2) \quad \hbox{for all} \quad i = 1, 2, \ldots, s^n. \end{array} $$
 Note in particular that this recursive rule is deterministic making the process stochastic only through its configuration
 at the bottom level.
 Galam \cite{galam_2008} assumes that nodes at the bottom level are independently in favor of a given opinion with a fixed probability.
 In contrast, we will assume that the configuration at the bottom level is a random permutation with a fixed number of nodes in favor
 of a given opinion. \vspace*{5pt}

\begin{figure}[t!]
\centering
 \includegraphics[width = 360pt]{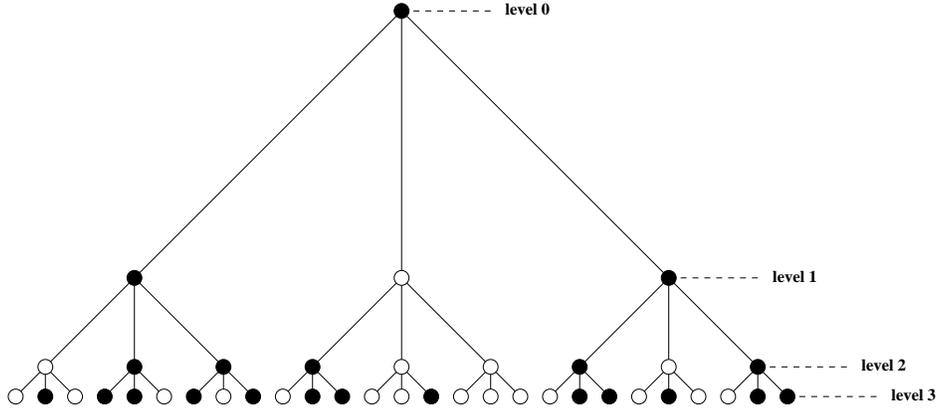}
\caption{\upshape Schematic illustration of the bottom-up hierarchical system with $s = 3$ and $N = 3$.
 Black dots represent individuals supporting opinion~$+1$ and white dots individuals supporting opinion~$-1$.}
\label{fig:voting-model}
\end{figure}


\noindent {\bf Non-spatial public debate model} --
 The second model under consideration in this paper is Galam's public debate model that examines the dynamics of opinion shifts.
 This process again depends on the same two parameters but now evolves in time.
 There is a population of~$N$ individuals each with either opinion~$+1$ or opinion~$-1$.
 At each time step, a random group of size~$s$, called discussion group, is chosen from the population, which results in all the
 individuals in the group adopting the same opinion.
 The updating rule considered in \cite{galam_2008} is again the majority rule:
 if there are opposing opinions in the discussion group, then the opinion with the majority of supporters dominates the other
 opinion causing the individuals who initially supported the minority opinion to change their opinion to the majority opinion.
 As previously, when the group size~$s$ is even, ties may occur, in which case a bias is introduced in favor of opinion~$-1$.
 We refer to Figure \ref{fig:debate-model} for a schematic representation of this process.
 In this paper, we will also consider another natural updating rule that we shall call proportional rule, which assumed that
 all the individuals in the group adopt opinion~$\pm 1$ with a probability equal to the fraction of supporters of this opinion
 in the group before discussion.
 To define these processes more formally, we now let
 $$ X_n (i) \quad \hbox{for} \quad n \in \N \quad \hbox{and} \quad i = 1, 2, \ldots, N $$
 be the opinion of individual $i$ at time $n$.
 In both processes, a set of $s$ individuals, say $B_s$, is chosen uniformly from the population at each time step.
 Under the majority rule, we set
 $$ \begin{array}{l} X_n (i) \ := \ \sign (\sum_{j \in B_s} X_{n - 1} (j) - 1/2) \quad \hbox{for all} \quad i \in B_s \end{array} $$
 while under the proportional rule, we set
 $$ \begin{array}{rclccl}
     X_n (i) & := & + 1 & \hbox{for all} \ i \in B_s & \hbox{with probability} & s^{-1} \ \sum_{j \in B_s} \ind \{X_n (j) = + 1 \} \vspace*{4pt} \\
             & := & - 1 & \hbox{for all} \ i \in B_s & \hbox{with probability} & s^{-1} \ \sum_{j \in B_s} \ind \{X_n (j) = - 1 \}. \end{array} $$
 In both processes, individuals outside $B_s$ are not affected by the discussion and the evolution rule is iterated until everyone
 in the population has the same opinion.
 We will see later that the process that keeps track of the number of individuals with opinion~$+1$ rather than the actual
 configuration is itself a discrete-time Markov chain.
 As for the bottom-up hierarchical system, we will assume that the configuration at time 0 has a fixed number of individuals in
 favor of a given opinion whereas Galam studied the (majority rule) public debate model under the assumption that initially
 individuals are independently in favor of a given opinion with a fixed probability. \vspace*{5pt}

\begin{figure}[t!]
\centering
 \includegraphics[width = 360pt]{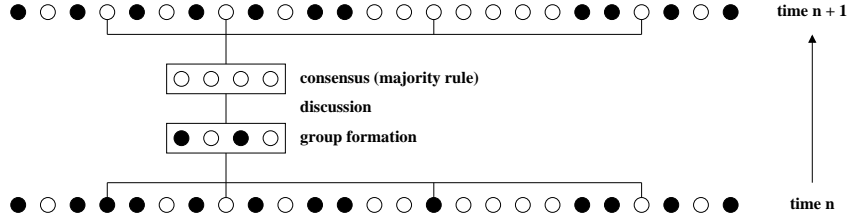}
\caption{\upshape One time step in the non-spatial public debate model with $s = 4$ and $N = 25$.
 Black dots represent individuals supporting opinion~$+1$ and white dots individuals supporting opinion~$-1$.}
\label{fig:debate-model}
\end{figure}


\noindent {\bf Spatial public debate model} --
 The third model studied in this paper is a spatial version of the public debate model introduced in \cite{lanchier_neufer_2013}.
 The spatial structure is represented by the infinite $d$-dimensional regular lattice.
 Each site of the lattice is occupied by one individual who is again characterized by their opinion: either opinion~$+1$ or
 opinion~$-1$.
 The population being located on a geometrical structure, space can be included by assuming that only individuals in the same
 neighborhood can interact.
 More precisely, we assume that the set of discussion groups is
 $$ x + B_s \quad \hbox{for} \quad x \in \Z^d \quad \hbox{where} \quad B_s := \{0, 1, \ldots, s - 1 \}^d $$
 Since the number of discussion groups is infinite and countable, the statement ``choosing a group uniformly at random'' is no
 longer well defined.
 Therefore, we define the process in continuous time using the framework of interacting particle systems assuming that discussion
 groups are updated independently at rate one, i.e., at the arrival times of independent Poisson processes with intensity one.
 The analysis in \cite{lanchier_neufer_2013} is concerned with the spatial model under the majority rule, whereas we focus on the
 proportional rule:
 all the individuals in the same discussion group adopt the same opinion with a probability equal to the fraction of
 supporters of this opinion in the group before discussion.
 Formally, the state of the process at time $t$ is now a function
 $$ \eta_t : \Z^d \ \longrightarrow \ \{-1, +1 \} $$
 with $\eta_t (x)$ denoting the opinion at time $t$ of the individual located at site $x$, and the dynamics of the process is
 described by the Markov generator
 $$ \begin{array}{rcl}
     Lf (\eta) & = &
    \displaystyle \sum_x \,\sum_{z \in B_s} \,s^{-d} \ \ind \{\eta (x + z) = + 1 \} \ [f (\tau_x^+ \eta) - f (\eta)] \vspace*{-4pt} \\ && \hspace*{60pt} + \
    \displaystyle \sum_x \,\sum_{z \in B_s} \,s^{-d} \ \ind \{\eta (x + z) = - 1 \} \ [f (\tau_x^- \eta) - f (\eta)] \end{array} $$
 where $\tau_x^+$ and $\tau_x^-$ are the operators defined on the set of configurations by
 $$ (\tau_x^+ \eta)(z) \ := \ \left\{\hspace*{-3pt} \begin{array}{ll}  + 1     & \hbox{for} \ z \in x + B_x \vspace*{2pt} \\
                                                                \eta (z) & \hbox{for} \ z \notin x + B_x \end{array} \right. \qquad
    (\tau_x^- \eta)(z) \ := \ \left\{\hspace*{-3pt} \begin{array}{ll}  - 1     & \hbox{for} \ z \in x + B_x \vspace*{2pt} \\
                                                                \eta (z) & \hbox{for} \ z \notin x + B_x. \end{array} \right. $$
 The first part of the generator indicates that, for each $x$, all the individuals in $x + B_s$ switch simultaneously to
 opinion~$+1$ at rate the fraction of individuals with opinion~$+1$ in the group.
 The second part gives similar transition rates for opinion~$-1$.
 In particular, basic properties of Poisson processes imply that each group is indeed updated at rate one according to the
 proportional rule.
 Note that the process no longer depends on $N$ since the population size is infinite but we will see that its behavior strongly
 depends on the spatial dimension $d$.


\section{Main results}
\label{sec:results}

\indent For the bottom-up hierarchical system and the non-spatial public debate model, the main problem is to determine the
 probability that a given opinion, say opinion~$+1$, wins, as a function of the density or number of individuals holding this
 opinion in the initial configuration.
 For the spatial public debate model, since the population is infinite, the time to reach a configuration in which all the
 individuals share the same opinion is almost surely infinite when starting from a configuration with infinitely many
 individuals of each type.
 In this case, the main problem is to determine whether opinions can coexist at equilibrium or not. \vspace*{5pt}


\noindent {\bf Galam's results} --
 Galam studied the bottom-up hierarchical system and the non-spatial public debate model under the majority rule.
 As previously explained, the assumption in \cite{galam_2008} about the initial configuration of each model is that individuals are
 independently in favor of opinion~$+1$ with some fixed probability.
 Under this assumption, the analysis is simplified because the probability of an individual being in favor of a given opinion at one level
 for the bottom-up hierarchical system or at one time step for the public debate model can be computed explicitly in a simple manner from
 its counterpart at the previous level or time step.
 More precisely, focusing on the bottom-up hierarchical system for concreteness, if $p_n$ is the common probability of any given
 individual being in favor of opinion~$+1$ at level~$n$ then the sequence $(p_n)$ can be computed recursively as follows:
 $$ p_n = Q_s (p_{n + 1}) \quad \hbox{where} \quad Q_s (X) \ := \ \sum_{j = s'}^s \ {s \choose j} \ X^j \,(1 - X)^{s - j} $$
 with $s' := \lceil (1/2)(s + 1) \rceil$.
 The probability that a given opinion wins the election can then be computed explicitly.
 For both models, in the limit as the population size tends to infinity, the problem reduces to finding the fixed points of the
 polynomial $Q_s$.
 When $s = 3$,
 $$ \begin{array}{rcl}
     Q_3 (X) - X & = & 3 \,X^2 (1 - X) + X^3 - X \vspace*{2pt} \\
                 & = & 3 \,X^2 (1 - X) + X \,(X - 1)(X + 1) \ = \ - X \,(X - 1)(2X - 1) \end{array} $$
 therefore $1/2$ is a fixed point.
 It follows that, with probability close to one when the population size is large, the winning opinion is the one that has initially the
 largest frequency of representants, a result that easily extends to all odd sizes.
 The case of even sizes is more intriguing.
 When the group size $s = 4$, we have
\begin{equation}
\label{eq:Q4}
  \begin{array}{rcl}
   Q_4 (X) - X & = & 4 \,X^3 (1 - X) + X^4 - X \vspace*{2pt} \\
               & = & - X \,(X - 1)(3 \,X^2 - X - 1) \ = \ - 3 \,X \,(X - 1)(X - c_-)(X - c_+) \end{array}
\end{equation}
 where the roots $c_-$ and $c_+$ are given by
 $$ c_- \ := \ \frac{1 - \sqrt{13}}{6} \ \approx \ - 0.434 \qquad \hbox{and} \qquad c_+ \ := \ \frac{1 + \sqrt{13}}{6} \ \approx \ 0.768. $$
 This implies that, when the population is large, the probability that opinion~$+1$ wins is near zero if the initial frequency of its
 representants is below $c_+ \approx 0.768$.
 It can be proved that the same result holds for the non-spatial public debate model when the population size is large.
 Because opinions are initially independent and of a given type with a fixed probability, the initial number of individuals with
 opinion~$+1$ is a binomial random variable, and the main reason behind the simplicity of Galam's results is that the dynamics of
 his models preserves this property: at any level/time, the number of individuals with opinion~$+1$ is again binomial.
 The first objective of this paper is to revisit Galam's results under the assumption that the initial number of individuals with
 opinion~$+1$ is a fixed deterministic number rather than binomially distributed.
 This assumption is more realistic for small populations but the analysis is also more challenging because
 the number of individuals with a given opinion in non-overlapping groups are no longer independent. \vspace*{5pt}


\noindent {\bf Bottom-up hierarchical system} --
 For the bottom-up hierarchical system, we start with a fixed deterministic number $x$ of individuals holding opinion~$+1$ at the
 bottom level.
 The main objective is then to determine the winning probability
\begin{equation}
\label{eq:proba-voting}
  \begin{array}{rcl}
     p_x (N, s) & := & \hbox{probability that opinion~$+1$ wins} \vspace*{4pt} \\
                & := & P \,(X_0 (1) = +1 \ | \ \card \{i : X_N (i) = + 1 \} = x) \end{array}
\end{equation}
 where $s$ is the group size and $N$ is the number of voting steps.
 Assuming that individuals holding the same opinion are identical, there are $s^N$ choose $x$ possible configurations at the bottom
 level of the system.
 To compute the probability \eqref{eq:proba-voting}, the most natural approach is to compute the number of such configurations that
 result in the election of candidate $+1$.
 This problem, however, is quite challenging so we use instead a different strategy.
 The main idea is to count configurations which are compatible with the victory of $+1$ going backwards in the hierarchy:
 we count the number of configurations at level one that result in the election of candidate $+1$, then the number of configurations
 at level two that result in any of these configurations at level one, and so on.
 To compute the number of such configurations, for each size-level pair $(s, n)$, we set
\begin{equation}
\label{eq:range}
  s' \ := \ \lceil (1/2)(s + 1) \rceil \quad \hbox{and} \quad I_{s, n} \ := \ \{0, 1, \ldots, (s' - 1)(s^n - x) + (s - s') \,x \}.
\end{equation}
 Then, for all $y \in I_{s, n}$, define
\begin{equation}
\label{eq:transition-s}
  c_n (s, x, s'x + y) \ = \ \sum_{z_0, \ldots, z_s} \ {s^n \choose x}^{-1} {s^n \choose z_0, z_1, \ldots, z_s} \ \prod_{j = 0}^s \ {s \choose j}^{z_j}
\end{equation}
 where the sum is over all $z_0, z_1, \ldots, z_s$ such that
 $$ z_0 + z_1 + \cdots + z_{s' - 1} \ = \ s^n - x \quad \hbox{and} \quad z_{s'} + z_{s' + 1} + \cdots + z_s \ = \ x $$
 and such that
 $$ \begin{array}{rcll}
    \sum_{j = 1}^{s' - 1} j \,(z_j + z_{s' + j}) & = & y & \hbox{if $s$ is odd} \vspace*{4pt} \\
    \sum_{j = 1}^{s' - 2} j \,(z_j + z_{s' + j}) + (s' - 1) \,z_{s' - 1} & = & y & \hbox{if $s$ is even}. \end{array} $$
 We will prove that the number of configurations with $s'x + y$ individuals holding opinion~$+1$ at level~$n + 1$ that result in a
 given configuration with $x$ individuals holding opinion~$+1$ at level~$n$ is exactly given by \eqref{eq:transition-s}.
 The fact that the evolution rules are deterministic also implies that different configurations at a given level cannot result from the
 same configuration at a lower level.
 In particular, the number of configurations at the bottom level that result in the victory of opinion~$+1$ can be deduced from
 a simple summation as in the proof of Chapman-Kolmogorov's equations in the theory of Markov chains.
 More precisely, we have the following theorem.
\begin{theor}[Bottom-up hierarchical system] --
\label{th:voting}
 For all $s \geq 3$, we have
\begin{equation}
\label{eq:chapman-kolmogorov}
  p_x (N, s) \ = \ {s^N \choose x}^{-1} \ \sum_{x_1 = 0}^s \ \sum_{x_2 = 0}^{s^2} \cdots \sum_{x_{N - 1} = 0}^{s^{N - 1}} \ \prod_{n = 1}^N \ c_n (s, x_{n - 1}, x_n)
\end{equation}
 where $x_0 = 1$ and $x_N = x$.
\end{theor}
 The expression for the probability \eqref{eq:chapman-kolmogorov} cannot be simplified but for any fixed parameter it can be
 computed explicitly.
 In the case of groups of size $s = 3$, \eqref{eq:transition-s} reduces to
\begin{equation}
\label{eq:transition-3}
  c_n (3, x, 2x + y) \ = \ \sum_{i + j = y} {x \choose i} {3^n - x \choose j} \ 3^{x - i + j} \quad \hbox{for} \quad y \in \{0, \ldots, 3^n \}
\end{equation}
 while in the case of groups of size $s = 4$, this reduces to
\begin{equation}
\label{eq:transition-4}
  c_n (4, x, 3x + y) \ = \ \sum_{i + 2j + k = y} {x \choose i} {4^n - x \choose j} {4^n - x - j \choose k} \ 4^{x - i + k} \ 6^j
\end{equation}
 for all $y \in \{0, 1, \ldots, 2 \times 4^n - x \}$.
 Figure \ref{fig:voting} shows the probabilities computed from \eqref{eq:chapman-kolmogorov}--\eqref{eq:transition-4} for different
 values of the number of levels and group size along with the corresponding probabilities under Galam's assumption.
 The figure reveals that the phase transition is sharper when starting from a fixed number rather than a binomially distributed
 number of individuals holding a given opinion.
 This aspect is more pronounced when the population size is small. \vspace*{5pt}

\begin{figure}[t!]
\centering
 \includegraphics[width = 360pt]{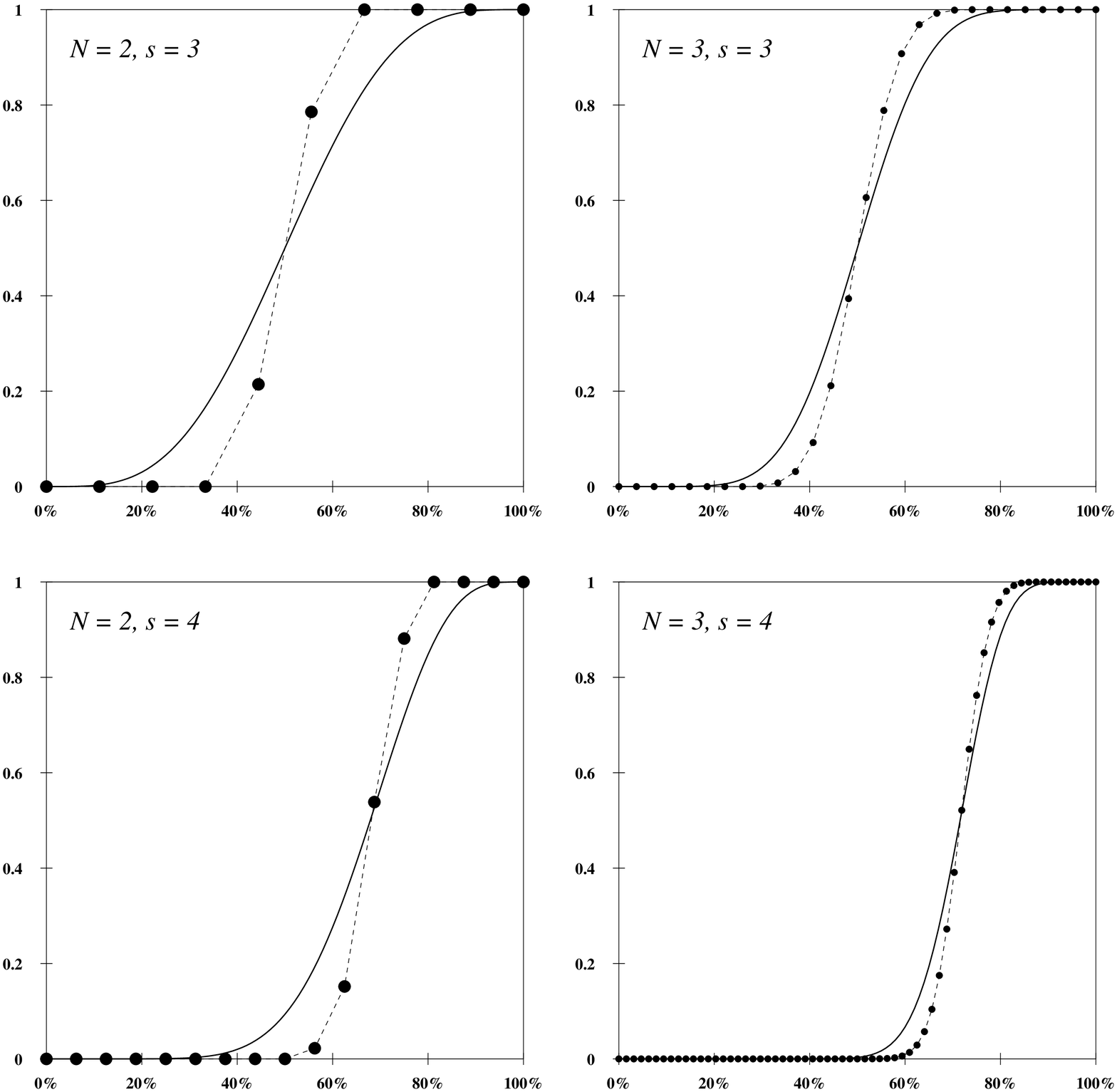}
\caption{\upshape Probability that opinion~$+1$ wins as a function of the initial density/number of its supporters at the bottom level
 of the bottom-up hierarchical system for different values of the number of levels and group size.
 The continuous black curve is the graph of the function $p \mapsto Q_s^N (p)$ corresponding to the winning probability when assuming
 that individuals at the bottom level hold independently opinion~$+1$ with probability $p$.
 The black dots are the probabilities computed from Theorem \ref{th:voting} when starting from a fixed number of individuals holding
 opinion~$+1$.}
\label{fig:voting}
\end{figure}




\begin{figure}[t!]
\centering
 \includegraphics[width = 360pt]{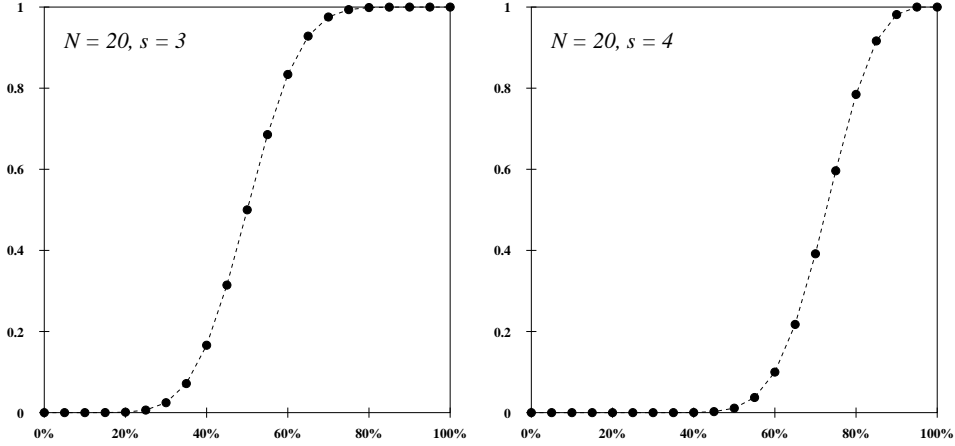}
\caption{\upshape Probability that opinion~$+1$ wins as a function of the initial number of its supporters in the non-spatial public
 debate model.
 The probabilities on the left ($s = 3$) are computed from the first part of Theorem \ref{th:debate-majority} whereas the ones on the
 right ($s = 4$) are computed recursively from a first-step analysis.}
\label{fig:debate-34}
\end{figure}

\noindent {\bf Non-spatial public debate model} --
 For the non-spatial public debate model, our main objective is again to determine the winning probability when starting from a fixed
 number of individuals holding opinion~$+1$.
 Since at each time step all the individuals are equally likely to be part of the chosen discussion group, the actual label on each
 individual is unimportant.
 In particular, we shall simply define $X_n$ as the number of individuals with opinion~$+1$ at time $n$ rather than the vector
 of opinions.
 With this new definition the winning probability can be written as
\begin{equation}
\label{eq:proba-debate}
  \begin{array}{rcl}
     p_x (N, s) & := & \hbox{probability that opinion~$+1$ wins} \vspace*{4pt} \\
                & := & P \,(X_n = N \ \hbox{for some} \ n > 0 \ | \ X_0 = x) \end{array}
\end{equation}
 where $s$ is the group size and $N$ is the total number of individuals.
 We start with the model under the majority rule.
 In this case, we have the following result.
\begin{theor}[non-spatial public debate model] --
\label{th:debate-majority}
 Under the majority rule,
 $$ p_x (N, 3) \ = \ 2^{- (N - 3)} \ \sum_{z = 0}^{x - 2} \ {N - 3 \choose z} \quad \hbox{for all} \ \ x \in \{2, \ldots, N - 2 \}. $$
 In addition, there exists $a_0 > 0$ such that, for all $\ep > 0$,
 $$ \begin{array}{rcll}
     p_x (N, 4) & \leq & \exp (- a_0 \ep N)     & \hbox{for all $N$ large and $x \in (0, (c_+ - 2 \ep) N)$} \vspace*{4pt} \\
     p_x (N, 4) & \geq & 1 - \exp (- a_0 \ep N) & \hbox{for all $N$ large and $x \in ((c_+ + 2 \ep) N, N)$}. \end{array} $$
\end{theor}
 Note that the first probability in the theorem can be re-written as
 $$ p_x (N, 3) \ = \ \frac{\card \{A : A \subset \{1, 2, \ldots, N - 3 \} \ \hbox{and} \ \card (A) \leq x - 2 \}}{\card \{A : A \subset \{1, 2, \ldots, N - 3 \} \}}. $$
 Unfortunately, we do not know why the winning probability has this simple combinatorial interpretation but this is what
 follows from our calculation which is based on a first-step analysis, a standard technique in the theory of Markov chains.
 This technique can also be used to determine the winning probabilities for larger $s$ recursively, which is how the
 right-hand side of Figure~\ref{fig:debate-34} is obtained, but for $s > 3$ the algebra becomes too complicated to get an
 explicit formula.
 Interestingly, the second part of the theorem shows that the critical threshold $c_+ \approx 0.768$ obtained under Galam's
 assumption appears again under our assumption on the initial configuration, though it comes from a different calculation.
 This result follows partly from an application of the optimal stopping theorem for supermartingales.
 Turning to the non-spatial public debate model under the proportional rule, first-step analysis is again problematic
 when the group size exceeds three.
 Nevertheless, the winning probabilities can be computed explicitly.
\begin{theor}[non-spatial public debate model] --
\label{th:debate-proportional}
 Under the proportional rule,
 $$ p_x (N, s) \ = \ x/N \quad \hbox{for all} \quad s > 1. $$
\end{theor}
 In words, under the proportional rule, the probability that opinion~$+1$ wins is simply equal to the initial fraction
 of individuals holding this opinion.
 As for the second part of Theorem \ref{th:debate-majority}, the proof of this result relies in part on an application of
 the optimal stopping theorem. \vspace*{5pt}


\noindent {\bf Spatial public debate model} --
 Contrary to the non-spatial public debate model, for the spatial version starting with a finite number of individuals
 with opinion~$+1$, the number of such individuals does not evolve according to a Markov chain because the actual location
 of these individuals matters.
 However, under the proportional rule, the auxiliary process that keeps track of the number of individuals with opinion
 $+1$ is a martingale with respect to the natural filtration of the spatial model.
 Since it is also integer-valued and the population is infinite, it follows from the martingale convergence theorem that
 opinion~$+1$ dies out with probability one.
 Therefore, to avoid trivialities, we return to Galam's assumption for the spatial model:
 we assume that individuals independently support opinion~$+1$ with probability $\theta \in (0, 1)$.
 Since the population is infinite, both opinions are present at any time, and the main objective is now to determine
 whether they can coexist at equilibrium.
 The answer depends on the spatial dimension $d$, as for the voter model \cite{clifford_sudbury_1973, holley_liggett_1975}.
\begin{theor}[spatial public debate model] --
\label{th:spatial-debate}
 Under the proportional rule,
\begin{itemize}
 \item the system clusters in $d \leq 2$, i.e.,
 $$ \begin{array}{l} \lim_{\,t \to \infty} \ P \,(\eta_t (x) \neq \eta_t (y)) \ = \ 0 \quad \hbox{for all} \quad x, y \in \Z^d. \end{array} $$
 \item both opinions coexist in $d \geq 3$, i.e., $\eta_t$ converges in distribution to an invariant measure in which
 there is a positive density of both opinions.
\end{itemize}
\end{theor}
 The proof relies on a certain duality relationship between the spatial model and coalescing random walks, just as for the
 voter model, though this relationship is somewhat hidden in the case of the public debate model.
 Before proving our theorems, we point out that the spatial public debate model under the majority rule has also been recently
 studied in \cite{lanchier_neufer_2013}.
 There, it is proved that the one-dimensional process clusters when the group size $s$ is odd whereas opinion~$-1$ invades the lattice
 and outcompetes the other opinion when the group size is even.
 It is also proved based on a rescaling argument that opinion~$-1$ wins in two dimensions when $s^2 = 2 \times 2 = 4$.

\section{Proof of Theorem \ref{th:voting} (bottom-up hierarchical system)}
\label{sec:voting}

\indent The main objective is to count the number of configurations at level~$N$ with $x$ individuals with opinion~$+1$
 that will deterministically result in the election of type~$+1$ president after $N$ consecutive voting steps.
 Even though the evolution rules of the voting system are deterministic (recall that the model is only stochastic through its
 random initial configuration), our approach is somewhat reminiscent of the theory of Markov chains.
 The idea is to reverse time by thinking of the type of the president at level zero as the initial state, and more generally
 the configuration at level~$n$ as the state at time $n$.
 In the theory of discrete-time Markov chains, the distribution at time $n$ given the initial state can be computed by looking
 at the $n$th power of the transition matrix, which keeps track of the probabilities of all possible sample paths that connect two
 particular states in $n$ time steps.
 To this extend, the right-hand side of \eqref{eq:chapman-kolmogorov} can be seen as the analog of the $n$th power of a transition
 matrix, or Chapman-Kolmogorov's equation, with however two exceptions.
 First, the expression \eqref{eq:chapman-kolmogorov} is more complicated because the number of individuals per level is not
 constant and therefore the evolution rules are not homogeneous in time.
 Second, and more importantly, the transition probability from $x \to z$ at time $n$ is replaced by an integer, namely
\begin{equation}
\label{eq:transition}
  \begin{array}{rcl}
   c_n (s, x, z) & := & \hbox{the number of configurations with $z$ individuals holding} \\
                 &    & \hbox{opinion~$+1$ at level~$n + 1$ that result in a given configuration} \\
                 &    & \hbox{with $x$ individuals holding opinion~$+1$ at level~$n$.} \end{array}
\end{equation}
 By thinking of the bottom-up hierarchical system going backwards in time, the question becomes:
 how many configurations with $x$ individuals of type~$+1$ at time/level~$N$ result from the initial configuration +1 at time/level zero,
 which corresponds to the victory of type~$+1$ president.
 To make the argument rigorous and prove \eqref{eq:chapman-kolmogorov}, we first define
 $$ \card X \ := \ \card \{i \in \{1, 2, \ldots, s^n \} : X (i) = 1 \} \quad \hbox{for all} \quad X \in \Lambda_{s^n} := \{-1, +1 \}^{s^n}. $$
 Recall that if $Z \in \Lambda_{s^{n + 1}}$ then the configuration $X$ at level~$n$ is given by
 $$ \begin{array}{l} X (i) \ := \ \sign (\sum_{j = 1}^s Z (s (i - 1) + j) - 1/2) \quad \hbox{for all} \quad i = 1, 2, \ldots, s^n. \end{array} $$
 which we write $Z \to X$.
 We also say that configuration $Z$ induces configuration $X$.
 More generally, we say that configuration $Z \in \Lambda_{s^m}$ induces configuration $X \in \Lambda_{s^n}$ if
 $$ \hbox{for all} \ i \in \{n, n + 1, \ldots, m - 1\}, \ \hbox{there exists} \ X^i \in \Lambda_{s^i} \ \hbox{such that} \ X^{i + 1} \to X^i $$
 where $X^m = Z$ and $X^n = X$, which we again write $Z \to X$.
 Finally, we let
 $$ c_n (s, X, z) \ := \ \card \{Z \in \Lambda_{s^{n + 1}} : Z \to X \ \hbox{and} \ \card Z = z \} \quad \hbox{for all} \quad X \in \Lambda_{s^n} $$
 denote the number of configurations with $z$ individuals of type~$+1$ at level~$n + 1$ that induce configuration $X$ at level~$n$.
 The first key is that $c_n (s, X, z)$ only depends on the number of type~$+1$ individuals in configuration $X$, which is proved in the following lemma.

\begin{lemma} --
\label{lem:card}
 Let $X, Y \in \Lambda_{s^n}$. Then,
 $$ \card X = \card Y \quad \hbox{implies that} \quad c_n (s, X, z) = c_n (s, Y, z). $$
\end{lemma}
\begin{proof}
 Since $\card X = \card Y$, there exists
 $$ \sigma \in \mathfrak S_{s^n} \quad \hbox{such that} \quad Y (i) = X (\sigma (i)) \quad \hbox{for} \ i = 1, 2, \ldots, s^n $$
 where $\mathfrak S_{s^n}$ denotes the permutation group.
 Using the permutation $\sigma$, we then construct an endomorphism on the set of configurations at level~$n + 1$ by setting
 $$ (\phi (Z))(s (i - 1) + j) \ := \ Z (s (\sigma (i) - 1) + j) \quad \hbox{for} \ i = 1, 2, \ldots, s^n \ \hbox{and} \ j = 1, 2, \ldots, s. $$
 In words, partitioning configurations into $s^n$ consecutive blocks of size $s$, we set
 $$ \hbox{$i$th block of $\phi (Z)$ \ := \ $\sigma (i)$th block of $Z$} \quad \hbox{for all} \quad i = 1, 2, \ldots, s^n. $$
 Now, we observe that
 $$ \begin{array}{rcl}
     Z \to X & \hbox{if and only if} & X (i) = \sign (\sum_{j = 1}^s Z (s (i - 1) + j) - 1/2) \ \ \hbox{for all} \ \ i \vspace*{4pt} \\
             & \hbox{if and only if} & X (\sigma (i)) = \sign (\sum_{j = 1}^s Z (s (\sigma (i) - 1) + j) - 1/2) \ \ \hbox{for all} \ \ i \vspace*{4pt} \\
             & \hbox{if and only if} & Y (i) = \sign (\sum_{j = 1}^s (\phi (Z))(s (i - 1) + j) - 1/2) \ \ \hbox{for all} \ \ i \vspace*{4pt} \\
             & \hbox{if and only if} & \phi (Z) \to Y. \end{array} $$
 Since in addition $\card Z = \card \phi (Z)$, which directly follows from the fact that $\phi (Z)$ is obtained from a permutation of the blocks of size $s$
 in $Z$, we deduce that
 $$ \phi \,(\{Z : Z \to X \ \hbox{and} \ \card Z = z \}) \,\subset \,\{Z : Z \to Y \ \hbox{and} \ \card Z = z \}. $$
 That is, for all $Z$ in the first set, $\phi (Z)$ is a configuration in the second set.
 To conclude, we observe that the function $\phi$ is an injection from the first set to the second set. Indeed,
 $$ \begin{array}{rcl}
     Z \neq Z' & \hbox{implies that} & Z (s (i - 1) + j) \neq Z' (s (i - 1) + j) \quad \hbox{for some} \ i, j \vspace*{4pt} \\
               & \hbox{implies that} & Z (s (\sigma (i) - 1) + j) \neq Z' (s (\sigma (i) - 1) + j) \quad \hbox{for some} \ i, j \vspace*{4pt} \\
               & \hbox{implies that} & (\phi (Z))(s (i - 1) + j) \neq (\phi (Z'))(s (i - 1) + j) \quad \hbox{for some} \ i, j \vspace*{4pt} \\
               & \hbox{implies that} & \phi (Z) \neq \phi (Z'). \end{array} $$
 The injectivity of $\phi$ implies that
 $$ \begin{array}{rcl}
     c_n (s, X, z) & = & \card \{Z \in \Lambda_{s^{n + 1}} : Z \to X \ \hbox{and} \ \card Z = z \} \vspace*{4pt} \\
                   & \leq & \card \{Z \in \Lambda_{s^{n + 1}} : Z \to Y \ \hbox{and} \ \card Z = z \} \ = \ c_n (s, Y, Z). \end{array} $$
 In particular, the lemma follows from the obvious symmetry of the problem.
\end{proof} \\ \\
 In view of Lemma \ref{lem:card}, for all $x \in \{0, 1, \ldots, s^n \}$, we can write
 $$ c_n (s, X, z) := c_n (s, x, z) \quad \hbox{for all} \quad X \in \Lambda_{s^n} \quad \hbox{with} \quad \card X = x. $$
 The interpretation of $c_n (s, x, z)$ is given in \eqref{eq:transition}.
 The next step to establish \eqref{eq:chapman-kolmogorov} is given by the following lemma which follows from the deterministic nature of the evolution rules.
\begin{lemma} --
\label{lem:disjoint}
 Let $X, Y \in \Lambda_{s^n}$. Then,
 $$ X \neq Y \quad \hbox{implies that} \quad \{Z \in \Lambda_{s^{n + 1}} : Z \to X \} \,\cap \,\{Z \in \Lambda_{s^{n + 1}} : Z \to Y \} \ = \ \varnothing. $$
\end{lemma}
\begin{proof}
 To begin with, observe that the assumption implies that
 $$ X (i) \neq Y (i) \quad \hbox{for some} \quad i = 1, 2, \ldots, s^n. $$
 In particular, if $Z \to X$ and $Z' \to Y$ then for this specific $i$ we have
 $$ \begin{array}{rcl}
     X (i) & = & \sign (\sum_{j = 1}^s Z (s (i - 1) + j) - 1/2) \vspace*{4pt} \\
           & \neq & \sign (\sum_{j = 1}^s Z' (s (i - 1) + j) - 1/2) \ = \ Y (i) \end{array} $$
 which in turn implies that
 $$ Z (s (i - 1) + j) \neq Z' (s (i - 1) + j) \quad \hbox{for some} \quad j = 1, 2, \ldots, s. $$
 In conclusion, $Z \neq Z'$.
 This completes the proof.
\end{proof} \\ \\
 Recalling \eqref{eq:transition} and using the fact that there is only one configuration at level zero in which type~$+1$ is president as well as
 the previous lemma, we deduce that the product
 $$ c_1 (s, x_0, x_1) \,c_2 (s, x_1, x_2) \ \cdots \ c_{N - 1} (s, x_{N - 2}, x_{N - 1}) \,c_N (s, x_{N - 1}, x_N) $$
 is the number of configurations with $x_N$ type~$+1$ individuals at level~$N$ that consecutively induce a configuration with $x_n$ type~$+1$ individuals
 at level~$n$.
 The number of configurations with $x$ type~$+1$ individuals at level~$N$ that result in the election of type~$+1$ is then obtained by setting
 $x_0 = 1$ and $x_N = x$ and by summing over all the possible values of the other $x_n$ which gives
 $$ \card \{X \in \Lambda_{s^N} : X \to (1) \ \hbox{and} \ \card X = x \} \ = \
    \sum_{x_1 = 0}^s \ \sum_{x_2 = 0}^{s^2} \cdots \sum_{x_{N - 1} = 0}^{s^{N - 1}} \ \prod_{n = 1}^N \ c_n (s, x_{n - 1}, x_n). $$
 As previously explained, this equation can be seen as the analog of Chapman-Kolmogorov's equation for time-heterogeneous Markov chains, though it
 represents a number of configurations rather than transition probabilities.
 Finally. since there are $s^N$ choose $x$ configurations with exactly $x$ type~$+1$ individuals at level~$N$, we deduce that the conditional
 probability that type~$+1$ is elected given that there are $x$ type~$+1$ individuals at the bottom of the hierarchy is
 $$ p_x (N, s) \ = \ {s^N \choose x}^{-1} \ \sum_{x_1 = 0}^s \ \sum_{x_2 = 0}^{s^2} \cdots \sum_{x_{N - 1} = 0}^{s^{N - 1}} \ \prod_{n = 1}^N \ c_n (s, x_{n - 1}, x_n). $$
 To complete the proof of the theorem, the last step is to compute $c_n (s, x, z)$.
 As a warming up, we start by proving equation \eqref{eq:transition-3}, the special case when $s = 3$.
\begin{lemma} --
\label{lem:number-3}
 For all $y \in \{0, 1, \ldots, 3^n \}$, we have
 $$ c_n (3, x, 2x + y) \ = \ \sum_{i + j = y} {x \choose i} {3^n - x \choose j} \ 3^{x - i + j}. $$
\end{lemma}
\begin{proof}
 Fix $X \in \Lambda_{s^n}$ with $\card X = x$.
 Assume that $Z \to X$ and let $z_j$ denote the number of blocks of size three with exactly $j$ type~$+1$ individuals, i.e.,
 $$ \begin{array}{l} z_j \ := \ \card \{i : \sum_{k = 1}^3 Z (3 (i - 1) + k) = j - (3 - j) \} \quad \hbox{for all} \quad j = 0, 1, 2, 3. \end{array} $$
 The fact that $\card X = x$ imposes
\begin{equation}
\label{eq:number-31}
  z_0 + z_1 \ = \ 3^n - x \quad \hbox{and} \quad z_2 + z_3 \ = \ x.
\end{equation}
 This implies that, for configuration $Z$,
\begin{itemize}
 \item there are $x$ choose $z_3$ permutations of the blocks with 2 or 3 type~$+1$ individuals, \vspace*{2pt}
 \item there are $3^n - x$ choose $z_1$ permutations of the blocks with 0 or 1 type~$+1$ individuals, \vspace*{2pt}
 \item there are 3 choose $j$ possible blocks of size three with $j$ type~$+1$ individuals.
\end{itemize}
 In particular, the number of $Z \to X$ with $z_j$ blocks with $j$ type~$+1$ individuals is
\begin{equation}
\label{eq:number-32}
 {x \choose z_3} {3^n - x \choose z_1} \ \prod_{j = 0}^3 \ {3 \choose j}^{z_j} \ = \ {x \choose z_3} {3^n - x \choose z_1} \ 3^{z_1 + z_2}.
\end{equation}
 Using again \eqref{eq:number-31} and the definition of $z_j$ also implies that
 $$ \begin{array}{rcl}
    \card Z & = & z_1 + 2 z_2 + 3 z_3 \vspace*{2pt} \\
            & = & z_1 + 2 \,(x - z_3) + 3 z_3 \ = \ 2x + z_1 + z_3 \,\in \,\{2x, 2x + 1, \ldots, 2x + 3^n \} \end{array} $$
 which gives the range for $y$ in the statement of the lemma and
 $$ y \ := \ (\card Z) - 2x \ = \ z_1 + z_3. $$
 This, together with \eqref{eq:number-32} and $z_1 + z_2 = z_1 + x - z_3$, finally gives
 $$ c_n (3, x, 2x + y) \ = \sum_{z_1 + z_3 = y} {x \choose z_3} {3^n - x \choose z_1} \ 3^{z_1 + z_2} \ = \sum_{z_1 + z_3 = y} {x \choose z_3} {3^n - x \choose z_1} \ 3^{x - z_3 + z_1}. $$
 This completes the proof.
\end{proof} \\ \\
 Following the same approach, we now prove the general case \eqref{eq:transition-s}.
\begin{lemma} --
\label{lem:number-s}
 For all $(s, n)$ and all $y \in I_{s, n}$ as defined in \eqref{eq:range}, we have
 $$ c_n (s, x, s'x + y) \ = \ \sum_{z_0, \ldots, z_s} \ {s^n \choose x}^{-1} {s^n \choose z_0, z_1, \ldots, z_s} \ \prod_{j = 0}^s \ {s \choose j}^{z_j} $$
 where the sum is over all $z_0, z_1, \ldots, z_s$ such that
 $$ z_0 + z_1 + \cdots + z_{s' - 1} \ = \ s^n - x \quad \hbox{and} \quad z_{s'} + z_{s' + 1} + \cdots + z_s \ = \ x $$
 and such that
 $$ \begin{array}{rcll}
    \sum_{j = 1}^{s' - 1} j \,(z_j + z_{s' + j}) & = & y & \hbox{if $s$ is odd} \vspace*{4pt} \\
    \sum_{j = 1}^{s' - 2} j \,(z_j + z_{s' + j}) + (s' - 1) \,z_{s' - 1} & = & y & \hbox{if $s$ is even}. \end{array} $$
\end{lemma}
\begin{proof}
 Again, we fix $X \in \Lambda_{s^n}$ with $\card X = x$, let $Z \to X$ and
 $$ \begin{array}{l} z_j \ := \ \card \{i : \sum_{k = 1}^s Z (s (i - 1) + k) = j - (s - j) \} \quad \hbox{for all} \quad j = 0, 1, \ldots, s. \end{array} $$
 The fact that $\card X = x$ now imposes
\begin{equation}
\label{eq:number-s1}
  z_0 + z_1 + \cdots + z_{s' - 1} \ = \ s^n - x \quad \hbox{and} \quad z_{s'} + z_{s' + 1} + \cdots + z_s \ = \ x.
\end{equation}
 This and the definition of $z_j$ imply that
 $$ \begin{array}{rcl}
    \card Z & = & (z_1 + 2 z_2 + \cdots + (s' - 1) \,z_{s' - 1}) \ + \ (s' z_{s'} + \cdots + s z_s) \vspace*{2pt} \\
            & = & (z_1 + 2 z_2 + \cdots + (s' - 1) \,z_{s' - 1}) \vspace*{2pt} \\ && \hspace*{25pt} + \
                   s' (x - z_{s' + 1} - \cdots - z_s) \ + \ ((s' + 1) \,z_{s' + 1} + \cdots + s \,z_s) \vspace*{2pt} \\
            & = &  s' x + z_1 + 2 \,z_2 + \cdots + (s' - 1) \,z_{s' - 1} + z_{s' + 1} + 2 \,z_{s' + 2} + \cdots + (s - s') \,z_s \vspace*{6pt} \\
    \card Z & \in & s' x + \{0, 1, \ldots, (s' - 1)(s^n - x) + (s - s') \,x \} \end{array} $$
 which gives the range for $y$.
 Rearranging the terms also gives
\begin{equation}
\label{eq:number-s2}
  \begin{array}{rcll}
   y \ := \ (\card Z) - s' x & = & \sum_{j = 1}^{s' - 1} j \,(z_j + z_{s' + j}) \quad & \hbox{if $s$ is odd} \vspace*{4pt} \\
                             & = & \sum_{j = 1}^{s' - 2} j \,(z_j + z_{s' + j}) + (s' - 1) \,z_{s' - 1} \quad & \hbox{if $s$ is even}. \end{array}
\end{equation}
 Now, using again \eqref{eq:number-s1}, we obtain:
 $$ \begin{array}{rcl}
    \begin{array}{c} \hbox{number of permutations of the blocks} \\ \hbox{with at least $s'$ type~$+1$ individuals} \end{array} & = &
    \displaystyle {x \choose z_{s'}, \ldots, z_s} \ := \ \frac{x!}{z_{s'}! \ \cdots \ z_s!} \vspace*{12pt} \\
    \begin{array}{c} \hbox{number of permutations of the blocks} \\ \hbox{with at most $s' - 1$ type~$+1$ individuals} \end{array} & = &
    \displaystyle {s^n - x \choose z_0, \ldots, z_{s' - 1}} \ := \ \frac{(s^n - x)!}{z_0! \ \cdots \ z_{s' - 1}!} \end{array} $$
 Since there are $s$ choose $j$ possible blocks of size $s$ with $j$ type~$+1$ individuals, the number of configurations with $z_j$ blocks
 with $j$ type~$+1$ individuals that induce $X$ is then
 $$ {x \choose z_{s'}, \ldots, z_s} {s^n - x \choose z_0, \ldots, z_{s' - 1}} \ \prod_{j = 0}^s \ {s \choose j}^{z_j} = \
    {s^n \choose x}^{-1} {s^n \choose z_0, z_1, \ldots, z_s} \ \prod_{j = 0}^s \ {s \choose j}^{z_j}. $$
 This implies that, for all suitable $y$,
 $$ c_n (s, x, s'x + y) \ = \ \sum_{z_0, \ldots, z_s} \ {s^n \choose x}^{-1} {s^n \choose z_0, z_1, \ldots, z_s} \ \prod_{j = 0}^s \ {s \choose j}^{z_j} $$
 where the sum is over all $z_0, z_1, \ldots, z_s$ such that \eqref{eq:number-s1} and \eqref{eq:number-s2} hold.
\end{proof}


\section{Proof of Theorems \ref{th:debate-majority} and \ref{th:debate-proportional} (non-spatial public debate model)}
\label{sec:debate}

\noindent This section is devoted to the proof of Theorems \ref{th:debate-majority} and \ref{th:debate-proportional} which deal
 with the nonspatial public debate model.
 There is no more hierarchical structure and the evolution rules are now stochastic.
 At each time step, $s$ distinct individuals are chosen uniformly at random to form a discussion group, which results in all the
 individuals within the group reaching a consensus.
 The new opinion is chosen according to either the majority rule or the proportional rule. \vspace*{5pt}


\noindent {\bf Majority rule and size 3} --
 In this case, the process can be understood by simply using a first-step analysis whose basic idea is to condition on all the
 possible outcomes of the first update and then use the Markov property to find a relationship among the winning probabilities for
 the process starting from different states.
 We point out that this approach is only tractable when $s = 3$ due to a small number of possible outcomes at each update.
\begin{lemma} --
\label{lem:debate-majority}
 Under the majority rule, we have
 $$ p_x (N, 3) \ = \ 2^{- (N - 3)} \ \sum_{z = 0}^{x - 2} \ {N - 3 \choose z} \quad \hbox{for all} \ \ x = 2, 3, \ldots, N - 2. $$
\end{lemma}
\begin{proof}
 The first step is to exhibit a relationship among the probabilities to be found by conditioning on all the possible outcomes of the first update.
 Recall that
 $$ p_x \ := p_x (N, 3) \ = \ P \,(X_n = N \ \hbox{for some} \ n \,| \,X_0 = x) $$
 and, for $x = 2, 3, \ldots, N - 2$, let $\mu_x := q_{-1} (x) / q_1 (x)$ where
\begin{equation}
\label{eq:left-right}
  q_j (x) \ := \ P \,(X_{n + 1} = x + j \,| \,X_n = x) \quad \hbox{for} \ j = -1, 1.
\end{equation}
 Conditioning on the possible values for $X_1$ and using the Markov property, we obtain
 $$ \begin{array}{rcl}
     p_x & = & P \,(X_n = N \ \hbox{for some} \ n \,| \,X_1 = x - 1) \ P \,(X_1 = x - 1 \,| \,X_0 = x) \vspace*{4pt} \\ && \hspace*{40pt} + \
               P \,(X_n = N \ \hbox{for some} \ n \,| \,X_1 = x) \ P \,(X_1 = x \,| \,X_0 = x) \vspace*{4pt} \\ && \hspace*{40pt} + \
               P \,(X_n = N \ \hbox{for some} \ n \,| \,X_1 = x + 1) \ P \,(X_1 = x + 1 \,| \,X_0 = x) \vspace*{4pt} \\ & = &
               q_{-1} (x) \ p_{x - 1} + (1 - q_{-1} (x) + q_1 (x)) \ p_x + q_1 (x) \ p_{x + 1}. \end{array} $$
 In particular, $q_1 (x) \,(p_{x + 1} - p_x) = q_{-1} (x) \,(p_x - p_{x - 1})$ so a simple induction gives
 $$ \begin{array}{rcl}
     p_{x + 1} - p_x & = & \mu_x \ (p_x - p_{x - 1}) \ = \mu_x \,\mu_{x - 1} \ (p_{x - 1} - p_{x - 2}) \vspace*{4pt} \\
                     & = & \cdots \ = \ \mu_x \,\mu_{x - 1} \ \cdots \ \mu_2 \ (p_2 - p_1) \ = \ \mu_x \,\mu_{x - 1} \ \cdots \ \mu_2 \ p_2. \end{array} $$
 Using again that $p_1 = 0$, it follows that
\begin{equation}
\label{eq:debate-majority-1}
  p_x \ = \ \sum_{z = 1}^{x - 1} \ (p_{z + 1} - p_z) \ = \ \bigg(1 + \ \sum_{z = 2}^{x - 1} \ \mu_2 \,\mu_3 \ \cdots \ \mu_z \bigg) \ p_2.
\end{equation}
 Now, using that $p_{N - 1} = 1$, we obtain
\begin{equation}
\label{eq:debate-majority-2}
  p_{N - 1} \ = \bigg(1 + \ \sum_{z = 2}^{N - 2} \ \mu_2 \,\mu_3 \ \cdots \ \mu_z \bigg) \ p_2 \ = \ 1.
\end{equation}
 Combining \eqref{eq:debate-majority-1} and \eqref{eq:debate-majority-2}, we deduce that
\begin{equation}
\label{eq:debate-majority-3}
  \begin{array}{rcl}
    p_x & = & \displaystyle \bigg(1 + \ \sum_{z = 2}^{x - 1} \ \mu_2 \,\mu_3 \ \cdots \ \mu_z \bigg) \ p_2 \vspace*{4pt} \\
        & = & \displaystyle \bigg(1 + \ \sum_{z = 2}^{x - 1} \ \mu_2 \,\mu_3 \ \cdots \ \mu_z \bigg) \bigg(1 + \ \sum_{z = 2}^{N - 2} \ \mu_2 \,\mu_3 \ \cdots \ \mu_z \bigg)^{-1}. \end{array}
\end{equation}
 To find an explicit expression for \eqref{eq:debate-majority-3}, the last step is to compute $q_{-1} (x)$ and $q_1 (x)$.
 Observing that these two probabilities are respectively the probability of selecting a group with one type~$+1$ individual and the
 probability of selecting a group with two type~$+1$ individuals, we get
 $$ q_{-1} (x) \ = \ {N \choose 3}^{-1} {x \choose 1} {N - x \choose 2} \quad \hbox{and} \quad q_1 (x) \ = \ {N \choose 3}^{-1} {x \choose 2} {N - x \choose 1}. $$
 This gives the following expression for the ratio:
 $$ \mu_x \ := \ \frac{q_{-1} (x)}{q_1 (x)} \ = \ \frac{x \,(N - x)(N - x - 1)}{x \,(x - 1)(N - x)} \ = \ \frac{N - x - 1}{x - 1} $$
 for $x = 2, 3, \ldots, N - 2$, and the following expression for the product:
\begin{equation}
\label{eq:debate-majority-4}
 \begin{array}{rcl}
    \displaystyle \mu_2 \,\mu_3 \ \cdots \ \mu_z & = &
    \displaystyle \frac{N - 3}{1} \ \frac{N - 4}{2} \ \cdots \ \frac{N - z - 1}{z - 1} \vspace*{8pt} \\ & = &
    \displaystyle \frac{(N - 3)!}{(z - 1)! \ (N - z - 2)!} \ = \ {N - 3 \choose z  - 1} \end{array}
\end{equation}
 for $z = 2, 3, \ldots, N - 2$.
 Finally, combining \eqref{eq:debate-majority-3} and \eqref{eq:debate-majority-4}, we obtain
 $$ \begin{array}{rcl}
     p_x & = & \displaystyle \bigg(1 + \ \sum_{z = 2}^{x - 1} \ {N - 3 \choose z - 1} \bigg) \bigg(1 + \ \sum_{z = 2}^{N - 2} \ {N - 3 \choose z - 1} \bigg)^{-1} \vspace*{8pt} \\
         & = & \displaystyle \bigg(\sum_{z = 0}^{N - 3} \ {N - 3 \choose z} \bigg)^{-1} \ \sum_{z = 0}^{x - 2} \ {N - 3 \choose z} \ = \
               \displaystyle 2^{- (N - 3)} \ \sum_{z = 0}^{x - 2} \ {N - 3 \choose z} \end{array} $$
 for all $x \in \{2, \ldots, N - 2 \}$.
 This completes the proof.
\end{proof} \\


\noindent {\bf Majority rule and size 4} --
 Increasing the common size of the discussion groups, a first-step analysis can again be used to find a recursive formula for the
 winning probabilities but the algebra becomes too messy to deduce an explicit formula.
 Instead, we prove lower and upper bounds for the winning probabilities using the optimal stopping theorem for supermartingales.
 To describe more precisely our approach, consider the transition probabilities
 $$ q_j (x) \ := \ P \,(X_{n + 1} - X_n = j \,| \,X_n = x) \quad \hbox{for} \quad j = -2, -1, 0, 1 $$
 as well as the new Markov chain $(Z_n)$ with transition probabilities
 $$ p (0, 0) \ = \ p (N, N) \ = \ 1 \quad \hbox{and} \quad p (x,x + j) \ = \ q_j (x) \,(q_1 (x) + q_{-1} (x) + q_{-2} (x))^{-1} $$
 for all $x = 1, 2, \ldots, N - 1$ and all $j = -2, -1, 1$.
 The process $(Z_n)$ can be seen as the random sequence of states visited by the public debate model until fixation.
 In particular,
 $$ p_x (N, 4) \ := \ P \,(X_n = N \ \hbox{for some} \ n > 0) \ = \ P \,(Z_n = N \ \hbox{for some} \ n > 0). $$
 The main idea of the proof is to first identify exponentials of the process $(Z_n)$ that are supermartingales and then apply the
 optimal stopping theorem to these processes.
 We start by proving that the drift of the Markov chain is either negative or positive depending on whether the number of individuals
 in favor of the $+ 1$ opinion is smaller or larger than $c_+ \approx 0.768$.
 In particular, we recover the critical threshold $c_+$ found by Galam using a different calculation.
\begin{lemma} --
\label{lem:drift}
 For all $\ep > 0$,
 $$ \begin{array}{rcll}
     E \,(Z_{n + 1} - Z_n \,| \,Z_n = x) & \leq & - \ (1/2)(\sqrt{13} - 1) \,\ep + O (N^{-1}) & \hbox{for all} \ \ x \in (0, (c_+ - \ep) N) \vspace*{4pt} \\
                                         & \geq & + \ \sqrt{13} \,\ep + O (N^{-1}) & \hbox{for all} \ \ x \in ((c_+ + \ep) N, N). \end{array} $$
\end{lemma}
\begin{proof}
 Observing that $q_{-2} (x)$ is the probability that a randomly chosen group of size 4 has two individuals in favor and two individuals
 against the $+ 1$ opinion, we obtain that
 $$ q_{-2} (x) \ = \ {N \choose 4}^{-1} \ \frac{x \,(x - 1)}{2} \ \frac{(N - x)(N - x - 1)}{2} \ = \ 6 \,c^2 \,(1 - c)^2 + O (N^{-1}) $$
 provided $x = \lfloor cN \rfloor$.
 Similarly, we show that
 $$ q_1 (x) \ = \ 4 \,c^3 \,(1 - c) + O (N^{-1}) \quad \hbox{and} \quad q_{-1} (x) \ = \ 4 \,c \,(1 - c)^3 + O (N^{-1}) $$
 from which it follows that
 $$ \begin{array}{rcl}
     q_1 (x) - q_{-1} (x) - 2 \,q_{-2} (x) & = & 4 \,c^3 \,(1 - c) - 4 \,c \,(1 - c)^3 - 12 \,c^2 \,(1 - c)^2 + O (N^{-1}) \vspace*{4pt} \\
                                           & = & 4 \,c \,(1 - c)(c^3 - c - 1) + O (N^{-1}) \vspace*{8pt} \\
     q_1 (x) + q_{-1} (x) + q_{-2} (x)     & = & 4 \,c^3 \,(1 - c) + 4 \,c \,(1 - c)^3 + 6 \,c^2 \,(1 - c)^2 + O (N^{-1}) \vspace*{4pt} \\
                                           & = & 2 \,c \,(1 - c)(c^2 - c + 2) + O (N^{-1}). \end{array} $$
 Taking the ratio of the previous two estimates leads to
 $$ E \,(Z_{n + 1} - Z_n \,| \,Z_n = x) \ = \ 6 \,(c - c_-)(c - c_+)(c^2 - c + 2)^{-1} + O (N^{-1}) $$
 from which we deduce that
 $$ E \,(Z_{n + 1} - Z_n \,| \,Z_n = x) \ \leq \ 3 \,(- c_-)(- \ep) + O (N^{-1}) \ = \ - (1/2)(\sqrt{13} - 1) \,\ep + O (N^{-1}) $$
 for all $x \in (0, (c_+ - \ep) N)$ and
 $$ E \,(Z_{n + 1} - Z_n \,| \,Z_n = x) \ \geq \ 3 \,(c_+ - c_-) \,\ep + O (N^{-1}) \ = \ \sqrt{13} \,\ep + O (N^{-1}) $$
 for all $x \in ((c_+ + \ep) N, N)$.
 This completes the proof.
\end{proof}
\begin{lemma} --
\label{lem:supermartingale}
 There exists $a_0 > 0$ such that
 $$ \begin{array}{rcl}
     E \,(\exp (a_0 Z_{n + 1}) - \exp (a_0 Z_n) \,| \,Z_n = x) \ \leq \ 0 & \hbox{for all} & x \in (0, (c_+ - \ep) N)  \vspace*{4pt} \\
     E \,(\exp (- a_0 Z_{n + 1}) - \exp (- a_0 Z_n) \,| \,Z_n = x) \ \leq \ 0 & \hbox{for all} & x \in ((c_+ + \ep) N, N) \end{array} $$
 for all $N$ sufficiently large.
\end{lemma}
\begin{proof}
 To begin with, we define the functions
 $$ \phi_x (a) \ := \ E \,(\exp (a Z_{n + 1}) - \exp (a Z_n) \,| \,Z_n = x). $$
 Differentiating then applying Lemma \ref{lem:drift}, we obtain
 $$ \begin{array}{rcl}
    \phi_x' (a) & = & E \,(Z_{n + 1} \,\exp (a Z_{n + 1}) - Z_n \,\exp (a Z_n) \,| \,Z_n = x) \vspace*{4pt} \\
    \phi_x' (0) & = & E \,(Z_{n + 1} - Z_n \,| \,Z_n = x) \ \leq \ - (1/2)(\sqrt{13} - 1) \,\ep + O (N^{-1}) \ < \ 0 \end{array} $$
 for all $x \in (0, (c_+ - \ep) N)$ and $N$ large.
 Since $\phi_x (0) = 0$, there is $a_+ > 0$ such that
 $$ \phi_x (a) \ \leq \ 0 \quad \hbox{for all} \ a \in (0, a_+) \ \hbox{and all} \ x \in (0, (c_+ - \ep) N). $$
 Differentiating $a \mapsto \phi_x (- a)$ and using Lemma \ref{lem:drift}, we also have
 $$ \phi_x (- a) \ \leq \ 0 \quad \hbox{for all} \ a \in (0, a_-) \ \hbox{and all} \ x \in ((c_+ + \ep) N, N) $$
 for some $a_- > 0$.
 In particular, for $a_0 := \min (a_+, a_-) > 0$,
 $$ \begin{array}{rcl}
      \phi_x (a_0) \ \leq \ 0 & \hbox{for all} & x \in (0, (c_+ - \ep) N) \vspace*{2pt} \\
    \phi_x (- a_0) \ \leq \ 0 & \hbox{for all} & x \in ((c_+ + \ep) N, N) \end{array} $$
 which, recalling the definition of $\phi_x$, is exactly the statement of the lemma.
\end{proof} \\ \\
 With Lemma \ref{lem:supermartingale} in hands, we are now ready to prove the upper and lower bounds for the winning probabilities
 using the optimal stopping theorem.
\begin{lemma} --
\label{lem:stopping-1}
 For all $\ep > 0$,
 $$ p_x (N, 4) \ \leq \ \exp (- a_0 \ep N) \quad \hbox{for all $N$ large and $x \in (0, (c_+ - 2 \ep) N)$}. $$
\end{lemma}
\begin{proof}
 First, we introduce the stopping times
 $$ \tau_0 \ := \ \inf \,\{n : Z_n = 0 \} \quad \hbox{and} \quad \tau_- \ := \ \inf \,\{n : Z_n > (c_+ - \ep) N \} $$
 as well as $T_- := \min (\tau_0, \tau_-)$.
 Since the process $\exp (a_0 Z_n)$ stopped at time $T_-$ is a supermartingale according to the first assertion in
 Lemma \ref{lem:supermartingale} and the stopping time $T_-$ is almost surely finite, the optimal stopping theorem implies that
\begin{equation}
\label{eq:stopping-1}
  E \,(\exp (a_0 Z_{T_-}) \,| \,Z_0 = x) \ \leq \ E \,(\exp (a_0 Z_0) \,| \,Z_0 = x) \ \leq \ \exp (a_0 (c_+ - 2 \ep) N)
\end{equation}
 for all $x \in (0, (c_+ - 2 \ep) N)$.
 In addition,
\begin{equation}
\label{eq:stopping-2}
 \begin{array}{rcl}
  E \,(\exp (a_0 Z_{T_-})) & = &
  E \,(\exp (a_0 Z_{T_-}) \,| \,T_- = \tau_0) \,P \,(T_- = \tau_0) \vspace*{4pt} \\ && \hspace*{50pt} + \ E \,(\exp (a_0 Z_{T_-}) \,| \,T_- = \tau_-) \,P \,(T_- = \tau_-) \vspace*{4pt} \\ & \geq &
  P \,(T_- = \tau_0) \ + \ \exp (a (c_+ - \ep) N) \,P \,(T_- \neq \tau_0) \vspace*{4pt} \\ & = &
  1 \ - \ (1 - \exp (a_0 (c_+ - \ep) N)) \,P \,(T_- \neq \tau_0).
\end{array}
\end{equation}
 Noticing that opinion +1 wins only if $T_- \neq \tau_0$ and combining \eqref{eq:stopping-1}--\eqref{eq:stopping-2}, we get
 $$ \begin{array}{rcl}
     p_x (N, 4) & \leq & P \,(T_- \neq \tau_0) \ \leq \ (\exp (a_0 (c_+ - 2 \ep) N) - 1)(\exp (a_0 (c_+ - \ep) N) - 1)^{-1} \vspace*{4pt} \\
                & \leq &   \exp (a_0 (c_+ - 2 \ep) N) \,(\exp (a_0 (c_+ - \ep) N))^{-1} \ = \ \exp (- a_0 \ep N). \end{array} $$
 for all $x \in (0, (c_+ - 2 \ep) N)$ and all $N$ sufficiently large.
\end{proof}
\begin{lemma} --
\label{lem:stopping-2}
 For all $\ep > 0$,
 $$ p_x (N, 4) \ \geq \ 1 - \exp (- a_0 \ep N) \quad \hbox{for all $N$ large and $x \in ((c_+ + 2 \ep) N, N)$}. $$
\end{lemma}
\begin{proof}
 This is similar to the proof of Lemma \ref{lem:stopping-1}.
 Let $T_+ := \min (\tau_N, \tau_+)$ where
 $$ \tau_N \ := \ \inf \,\{n : Z_n = N \} \quad \hbox{and} \quad \tau_+ \ := \ \inf \,\{n : Z_n < (c_+ + \ep) N \} $$
 and apply Lemma \ref{lem:supermartingale} and the optimal stopping theorem to obtain
\begin{equation}
\label{eq:stopping-3}
  E \,(\exp (- a_0 Z_{T_+}) \,| \,Z_0 = x) \ \leq \ E \,(\exp (- a_0 Z_0) \,| \,Z_0 = x) \ \leq \ \exp (- a_0 (c_+ + 2 \ep) N)
\end{equation}
 for all $x \in ((c_+ + 2 \ep) N, N)$.
 Moreover,
\begin{equation}
\label{eq:stopping-4}
 \begin{array}{l}
  E \,(\exp (- a_0 Z_{T_+})) \ = \
  E \,(\exp (- a_0 Z_{T_+}) \,| \,T_+ = \tau_N) \,P \,(T_+ = \tau_N) \vspace*{4pt} \\ \hspace*{125pt} + \ E \,(\exp (- a_0 Z_{T_+}) \,| \,T_+ = \tau_+) \,P \,(T_+ = \tau_+) \vspace*{4pt} \\ \hspace*{50pt} \geq \
    \exp (- a_0 N) \,P \,(T_+ = \tau_N) + \exp (- a_0 (c_+ + \ep) N) \,P \,(T_+ \neq \tau_N) \vspace*{4pt} \\ \hspace*{50pt} = \
    \exp (- a_0 (c_+ + \ep) N) + \ (\exp (- a_0 N) - \exp (- a_0 (c_+ + \ep) N)) \,P \,(T_+ = \tau_N). \end{array}
\end{equation}
 Combining \eqref{eq:stopping-3}--\eqref{eq:stopping-4} and using that opinion +1 wins if $T_+ = \tau_N$, we get
 $$ \begin{array}{rcl}
     p_x (N, 4) & \geq & P \,(T_+ = \tau_N) \vspace*{4pt} \\
                & \geq & (\exp (- a_0 (c_+ + \ep) N) - \exp (- a_0 (c_+ + 2 \ep) N)) \vspace*{4pt} \\ && \hspace*{80pt} (\exp (- a_0 (c_+ + \ep) N)- \exp (- a_0 N))^{-1} \vspace*{4pt} \\
                & \geq & (\exp (- a_0 (c_+ + \ep) N) - \exp (- a_0 (c_+ + 2 \ep) N)) \,\exp (a_0 (c_+ + \ep) N) \vspace*{4pt} \\
                & = & 1 - \exp (- a_0 \ep N) \end{array} $$
 for all $x \in ((c_+ + 2 \ep) N, N)$ and all $N$ sufficiently large.
\end{proof} \\


\noindent {\bf Proportional rule} --
 We now prove Theorem \ref{th:debate-proportional}, which deals with the non-spatial public debate model under the proportional rule.
 To begin with, we introduce the transition probabilities
 $$ \begin{array}{rcll}
     r_j (x) & := & P \,(X_{n + 1} - X_n = j \ | \,X_n = x)   & \hbox{for all} \ \ j \geq 0 \vspace*{4pt} \\
     l_j (x) & := & P \,(X_{n + 1} - X_n = - j \ | \,X_n = x) & \hbox{for all} \ \ j \geq 0 \end{array} $$
 where $r$ and $l$ stand for right and left, respectively.
 As previously, a first-step analysis does not allow to find an explicit expression for the winning probabilities, but the result
 can be deduced from the optimal stopping theorem observing that the number of individuals in favor of a given opinion is a martingale
 with respect to the natural filtration of the process.
\begin{lemma} --
\label{lem:debate-proportional}
 Under the proportional rule, we have $p_x (N, s) = x/N$.
\end{lemma}
\begin{proof}
 Since an update can only result in an increase of the number of +1 individuals by $j$ if the discussion group selected
 has exactly $s - j$ individuals in favor of +1, we have
\begin{equation}
\label{eq:right}
  r_j (x) \ = \ {N \choose s}^{-1} {x \choose s - j} {N - x \choose j} \bigg(\frac{s - j}{s} \bigg)
\end{equation}
 for all $j \in I_x := \{\max (0, s - x), \ldots, \min (s, N - x) \}$. Similarly,
\begin{equation}
\label{eq:left}
  l_j (x) \ = \ {N \choose s}^{-1} {x \choose j} {N - x \choose s - j} \bigg(\frac{s - j}{s} \bigg)
\end{equation}
 for all $j \in J_x := \{\max (0, s - (N - x)), \ldots, \min (s, x) \}$. Now, let
 $$ m (x) \ := \ \max (0, s - x) \quad \hbox{and} \quad M (x) \ := \ \min (s, x) $$
 and observe that
 $$ \begin{array}{rcl}
     s - m (x) & = & s - \max (0, s - x) \ = \ s + \min (0, x - s) \ = \ \min (s, x) \ = \ M (x) \vspace*{4pt} \\
     s - M (N - x) & = & s - \min (s, N - x) \ = \ s + \max (- s, - (N - x)) \vspace*{2pt} \\ & = & \max (0, s - (N - x)) \ = \ m (N - x). \end{array} $$
 This shows that $J_x = s - I_x$ for all $x$.
 In particular, using the transformation $j \mapsto s - j$ and recalling the expression of the two conditional probabilities \eqref{eq:right}--\eqref{eq:left}, we obtain
 $$ \begin{array}{rcl}
    \displaystyle \sum_{j \in I_x} \  j \,r_j (x) & = &
    \displaystyle \sum_{j \in I_x} \  j \ {N \choose s}^{-1} {x \choose s - j} {N - x \choose j} \bigg(\frac{s - j}{s} \bigg) \vspace*{4pt} \\ & = &
    \displaystyle \sum_{j \in J_x} \ (s - j) \ {N \choose s}^{-1} {x \choose j} {N - x \choose s - j} \bigg(\frac{j}{s} \bigg) \vspace*{4pt} \\ & = &
    \displaystyle \sum_{j \in J_x} \  j \ {N \choose s}^{-1} {x \choose j} {N - x \choose s - j} \bigg(\frac{s - j}{s} \bigg) \ = \
    \displaystyle \sum_{j \in J_x} \  j \,l_j (x) \end{array} $$
 which gives the conditional expectation
 $$ E \,(X_{n + 1} - X_n \ | \,X_n = x) \ = \
      \displaystyle \sum_{j \in I_x} \ j \,r_j (x) \ - \ \sum_{j \in J_x} \ j \,l_j (x) \ = \ 0. $$
 This shows that the process $(X_n)$ is a martingale. Now, let
 $$ T_+ \ := \ \inf \,\{n : X_n = N \} \quad \hbox{and} \quad T_- \ := \ \inf \,\{n : X_n = 0 \} $$
 and observe that the stopping time $T := \min (T_+, T_-)$ is almost surely finite.
 Since in addition the process is bounded, an application of the optimal stopping theorem implies that
 $$ \begin{array}{rcl}
     E \,(X_T \,| \,X_0 = x) & = &
     E \,(X_0 \,| \,X_0 = x) \ = \ x \vspace*{4pt} \\ & = &
     N \times P \,(T = T_+) + 0 \times P \,(T = T_-) \ = \ N \,p_x (N, s) \end{array} $$
 from which it follows that $p_x (N, s) = x / N$.
\end{proof}


\section{Proof of Theorem \ref{th:spatial-debate} (spatial public debate model)}
\label{sec:space}

\indent To conclude, we study the spatial version of the public debate model introduced in \cite{lanchier_neufer_2013} but replacing
 the majority rule with the proportional rule.
 The key to our analysis is similar to the approach used in previous works \cite{clifford_sudbury_1973, holley_liggett_1975} about
 the voter model.
 The idea is to construct the process from a so-called Harris' graphical representation and then use the resulting graphical
 structure to exhibit a relationship between the process and a system of coalescing random walks. \vspace*{5pt}

\noindent {\bf Graphical representations} --
 We first give a possible graphical representation from which the spatial public debate model can be constructed starting
 from any initial configuration.
 Though natural, this graphical representation does not allow to derive a useful duality relationship between the process and
 coalescing random walks.
 We then introduce an alternative way to construct the process leading to such a duality relationship.
 Recall that
 $$ \{x + B_s : x \in \Z^d \} \quad \hbox{where} \quad B_s := \{0, 1, \ldots, s - 1 \}^d $$
 represents the collection of discussion groups.
 Each of these groups is updated in continuous time at rate one, i.e., at the arrival times of independent Poisson processes with
 intensity one.
 In addition, since the new opinion of the group after an update is chosen to be +1 with probability the fraction of +1 individuals
 in the group just before the update, the new opinion can be determined by comparing the fraction of +1 with a uniform random
 variable over the unit interval.
 In particular, a natural way to construct the spatial public debate model graphically is to
\begin{itemize}
 \item let $T_n (x) :=$ the $n$th arrival time of a Poisson process with rate one and \vspace*{4pt}
 \item let $U_n (x) :=$ a uniform random variable over the interval $(0, 1)$
\end{itemize}
 for all $x \in \Z^d$ and $n > 0$.
 At time $t := T_n (x)$, all the individuals in $x + B_s$ are simultaneously updated as a result of a discussion and we set
\begin{equation}
\label{eq:rule-1}
  \begin{array}{l} \eta_t (y) \ := \ 2 \times \ind \{U_n (x) < s^{-d} \ \sum_{z \in x + B_s} \ind \{\eta_{t-} (z) = +1 \} \} - 1 \quad \hbox{for all} \quad y \in x + B_s \end{array}
\end{equation}
 while the configuration outside $x + B_s$ stays unchanged.
 An idea of Harris \cite{harris_1972} implies that the process starting from any initial configuration can be constructed using
 this rule.
 We now construct another process $(\xi_t)$ with the same state space as follows:
 the times at which individuals in the same discussion group interact are defined as above from the same collection of independent
 Poisson processes, but to determine the outcome of the discussion we now
\begin{itemize}
 \item let $W_n (x) :=$ a uniform random variable over the set $x + B_s$
\end{itemize}
 for all $x \in \Z^d$ and $n > 0$.
 At time $t := T_n (x)$, all the individuals in $x + B_s$ are simultaneously updated as a result of a discussion and we set
\begin{equation}
\label{eq:rule-2}
  \begin{array}{l} \xi_t (y) \ := \ \xi_{t-} (W_n (x)) \quad \hbox{for all} \quad y \in x + B_s \end{array}
\end{equation}
 while the configuration outside $x + B_s$ stays unchanged.
 The next lemma, whose proof is simply based on a re-writing of events under consideration, shows that both
 rules \eqref{eq:rule-1}--\eqref{eq:rule-2} define in fact the same process:
 the processes $(\eta_t)$ and $(\xi_t)$ are stochastically equal.
\begin{lemma} --
\label{lem:equivalent}
 Both constructions \eqref{eq:rule-1}--\eqref{eq:rule-2} are equivalent:
 $$ \eta_{t-} = \,\xi_{t-} \quad \hbox{implies} \quad P \,(\eta_t (x) = 1) \ = \ P \,(\xi_t (x) = 1) \ \ \hbox{for all} \ \ x \in \Z^d. $$
\end{lemma}
\begin{proof}
 This is only nontrivial for pairs $(x, t) \in \Z^d \times \R_+$ such that
 $$ t \ := \ T_n (z) \ \ \hbox{and} \ \ x \in z + B_s \quad \hbox{for some} \quad (z, n) \in \Z^d \times \N^*. $$
 In this case, we have
 $$ \begin{array}{rcl}
     P \,(\xi_t (x) = 1) & = &
     P \,(\xi_{t-} (W_n (z)) = 1) \vspace*{4pt} \\ & = &
     P \,(W_n (z) \in \{y \in z + B_s : \xi_{t-} (y) = 1 \}) \vspace*{4pt} \\ & = &
       \card \{y \in z + B_s : \xi_{t-} (y) = 1 \} / \card (z + B_s) \vspace*{4pt} \\ & = &
        s^{-d} \ \sum_{y \in z + B_s} \,\ind \{\xi_{t-} (y) = 1 \} \vspace*{4pt} \\ & = &
     P \,(U_n (z) < s^{-d} \ \sum_{y \in z + B_s} \,\ind \{\xi_{t-} (y) = 1 \}) \vspace*{4pt} \\ & = &
     P \,(U_n (z) < s^{-d} \ \sum_{y \in z + B_s} \,\ind \{\eta_{t-} (y) = 1 \}) \ = \ P \,(\eta_t (x) = 1). \end{array} $$
 This completes the proof of the lemma.
\end{proof} \\

\begin{figure}[t!]
\centering
\includegraphics[width=360pt]{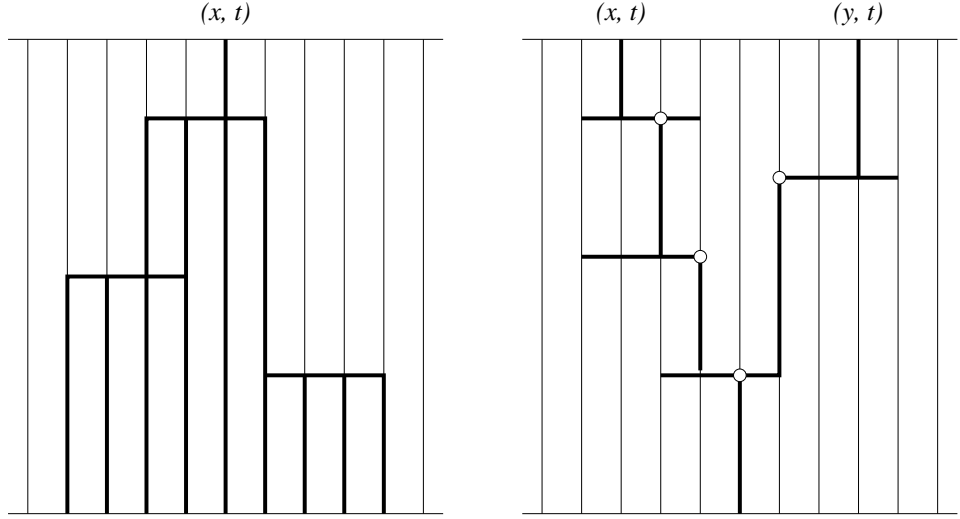}
\caption{\upshape{Picture of the graphical representation and set of ancestors.
 In both pictures, $s = 4$ and the times are which discussion groups are updated (time goes up) are represented by horizontal line segments
 while the set of ancestors is represented by vertical line segments.
 The left-hand and right-hand pictures give respectively an illustration of the set-valued process \eqref{eq:dual-eta} and an illustration
 of the dual process \eqref{eq:dual-xi} starting from $A = \{x, y \}$.
 The open circles on the right-hand side correspond to the value of the uniform $W$ random variables.}}
\label{fig:dual}
\end{figure}


\noindent {\bf Duality with coalescing random walks} --
 The duality relationship between the voter model and coalescing random walks results from keeping track of the ancestors of
 different space-time points going backwards in time through the graphical representation.
 In the case of the public debate model $\eta_{\cdot}$, the opinion of an individual just after an interaction depends on the opinion of all
 the individuals in the corresponding discussion group just before the interaction.
 Therefore, to define the set of ancestors of a given space-time point, we draw an arrow
 $$ z_1 \ \to \ z_2 \quad \hbox{at time} \quad t := T_n (z) \quad \hbox{for all} \quad z_1, z_2 \in z + B_s \ \hbox{and} \ (z, n) \in \Z^d \times \N^* $$
 to indicate that the opinion at $(z_2, t)$ depends on the opinion at $(z_1, t-)$, and say that there is a $\eta$-path connecting
 two space-time points, which we write
 $$ (y, t - s) \to_{\eta} (x, t) \quad \hbox{for} \quad x, y \in \Z^d \ \hbox{and} \ s, t > 0, $$
 whenever there are sequences of times and spatial locations
 $$ t - s < s_1 < s_2 < \cdots < s_{n - 1} < t \qquad \hbox{and} \qquad z_1 := y, z_2, \ldots, z_n := x \in \Z^d $$
 such that there is an arrow
 $$ z_j \ \to \ z_{j + 1} \ \ \hbox{at time} \ \ s_j \ \ \hbox{for} \ \ j = 1, 2, \ldots, n - 1. $$
 The set of ancestors of $(x, t)$ at time $t - s$ is then encoded in the set-valued process
\begin{equation}
\label{eq:dual-eta}
 \hat \eta_s (x, t) \ := \ \{y \in \Z^d : \hbox{$(y, t - s) \to_{\eta} (x, t)$} \}.
\end{equation}
 Note that the opinion at $(x, t)$ can be deduced from the graphical representation of $\eta_{\cdot}$ and the initial opinion at sites
 that belong to $\hat \eta_t (x, t)$.
 Note also that the process \eqref{eq:dual-eta} grows linearly going backwards in time, i.e., increasing $s$.
 See the left-hand side of Figure \ref{fig:dual} for a picture.
 This makes the process $\eta_{\cdot}$ mathematically intractable to prove clustering and coexistence.
 To establish the connection between the spatial process and coalescing random walks, we use instead the other, mathematically equivalent,
 version $\xi_{\cdot}$ of the spatial public debate model.
 For this version, the opinion of an individual just after an interaction depends on the opinion of only one individual in the corresponding discussion
 group just before the interaction.
 The location of this individual is given by the value of the uniform $W$ random variables.
 Therefore, to define the set of ancestors of a given space-time point, we now draw an arrow
 $$ W_n (z) \ \to \ z' \quad \hbox{at time} \quad t := T_n (z) \quad \hbox{for all} \quad z' \in z + B_s \ \hbox{and} \ (z, n) \in \Z^d \times \N^* $$
 to indicate that the opinion at $(z', t)$ depends on the opinion at $(W_n (z), t-)$.
 We then define $\xi$-paths, which we now write $\to_{\xi}$, as previously but using this new random set of arrows.
 The set of ancestors of $(x, t)$ at time $t - s$ is now encoded in
\begin{equation}
\label{eq:dual-xi}
  \hat \xi_s (x, t) \ := \ \{y \in \Z^d : \hbox{$(y, t - s) \to_{\xi} (x, t)$} \}.
\end{equation}
 More generally, for $A \subset \Z^d$ finite, we define the dual process starting at $(A, t)$ as
 $$ \begin{array}{rcl}
    \hat \xi_s (A, t) & := & \{y \in \Z^d : y \in \hat \xi_s (x, t) \ \hbox{for some} \ x \in A \} \vspace*{4pt} \\
                      & := & \{y \in \Z^d : (y, t - s) \to_{\xi} (x, t) \ \hbox{for some} \ x \in A \}. \end{array} $$
 See the right-hand side of Figure \ref{fig:dual} for a picture.
 Note that \eqref{eq:dual-xi} is reduced to a singleton for all times $s \in (0, t)$ and that we have the duality relationship
\begin{equation}
\label{eq:duality}
  \xi_t (x) \ = \ \xi_{t - s} (Z_s (x)) \ = \ \xi_0 (Z_t (x))\quad \hbox{for all} \quad s \in (0, t)
\end{equation}
 where $Z_s (x) := \hat \xi_s (x, t)$.
 In the next lemma, we prove that $Z_s (x)$ is a symmetric random walk, which makes the dual process itself a system of coalescing symmetric random
 walks with one walk starting from each site in the finite set $A$.
\begin{lemma} --
\label{lem:lineage}
 The process $Z_s (x) := \hat \xi_s (x, t)$ is a symmetric random walk.
\end{lemma}
\begin{proof}
 By construction of the dual process, for $t - s := T_n (z)$,
 $$ \begin{array}{rclcl}
     Z_s (x) \ := \ \hat \xi_s (x, t) & = & Z_{s-} (x) & \hbox{when} & Z_{s-} (x) \notin z + B_s  \vspace*{4pt} \\
                                      & = &  W_n (z)   & \hbox{when} & Z_{s-} (x) \in    z + B_s. \end{array} $$
 Since in addition discussion groups are updated at rate one and
 $$ P \,(W_n (z) = y) = s^{-d} \quad \hbox{for all} \quad y \in z + B_s $$
 we obtain the following transition rates:
\begin{equation}
\label{eq:lineage-1}
 \begin{array}{l}
  \lim_{\,h \to 0} \ h^{-1} \,P \,(Z_{s + h} (x) = y + w \ | \ Z_s (x) = y) \vspace*{4pt} \\ \hspace*{25pt} = \
   s^{-d} \ \card \{z \in \Z^d : y \in z + B_s \ \hbox{and} \ y + w \in z + B_s \}. \end{array}
\end{equation}
 In addition, since for all $w \in \Z^d$ the translation operator $y \mapsto y + w$ is a one-to-one correspondence from the set
 of discussion groups to itself and since
 $$ y, y + w \in z + B_s \quad \hbox{if and only if} \quad y - w, y \in (z - w) + B_s $$
 we have the equality
\begin{equation}
\label{eq:lineage-2}
 \begin{array}{l}
  \card \{z \in \Z^d : y \in z + B_s \ \hbox{and} \ y + w \in z + B_s \} \vspace*{4pt} \\ \hspace*{25pt} = \
  \card \{z \in \Z^d : y \in z + B_s \ \hbox{and} \ y - w \in z + B_s \}. \end{array}
\end{equation}
 Combining \eqref{eq:lineage-1}--\eqref{eq:lineage-2}, we conclude that
 $$ \begin{array}{l}
    \lim_{\,h \to 0} \ h^{-1} \,P \,(Z_{s + h} (x) = y + w \ | \ Z_s (x) = y) \vspace*{4pt} \\ \hspace*{25pt} = \
    \lim_{\,h \to 0} \ h^{-1} \,P \,(Z_{s + h} (x) = y - w \ | \ Z_s (x) = y) \end{array} $$
 for all $y, w \in \Z^d$, which completes the proof.
\end{proof} \\ \\
 In fact, some basic geometry shows that
 $$ \begin{array}{l}
    \lim_{\,h \to 0} \ h^{-1} \,P \,(Z_{s + h} (x) = y + w \ | \ Z_s (x) = y) \ = \ s^{-d} \ \displaystyle \prod_{j = 1}^d \ (s - |w_j|) \end{array} $$
 where $w_j$ is the $j$th coordinate of the vector $w$.
 With Lemma \ref{lem:equivalent}, which shows that $\xi_{\cdot}$ is indeed the spatial public debate model, and the previous lemma
 in hands, the rest of the proof of the theorem follows the lines of the corresponding result for the voter
 model \cite{clifford_sudbury_1973, holley_liggett_1975}.
 Since it is short, we briefly recall the main ideas in the next two lemmas that deal with the clustering part and the coexistence part of the
 theorem, respectively.
\begin{lemma} --
\label{lem:clustering}
 Assume that $d \leq 2$. Then,
 $$ \begin{array}{l} \lim_{\,t \to \infty} \ P \,(\xi_t (x) \neq \xi_t (y)) \ = \ 0 \quad \hbox{for all} \ x, y \in \Z^d. \end{array} $$
\end{lemma}
\begin{proof}
 Since $Z_s (x)$ and $Z_s (y)$ evolve according to independent random walks run at rate one until they coalesce, the difference between the random
 walks $Z_s (x) - Z_s (y)$ is a continuous-time symmetric random walk run at rate two absorbed at site zero.
 Since this random walk has in addition a finite range of interactions, it is recurrent in one and two dimensions, hence
 $$ \begin{array}{l} \lim_{\,t \to \infty} P \,(Z_t (x) \neq Z_t (y)) \ = \
                     \lim_{\,t \to \infty} P \,(Z_s (x) - Z_s (y) \neq 0 \ \hbox{for all} \ s < t) \ = \ 0. \end{array} $$
 By the duality relationship \eqref{eq:duality}, we conclude that
 $$ \begin{array}{l} \lim_{\,t \to \infty} P \,(\xi_t (x) \neq \xi_t (y)) \ \leq \ \lim_{\,t \to \infty} P \,(Z_t (x) \neq Z_t (y)) \ = \ 0. \end{array} $$
 This completes the proof.
\end{proof}
\begin{lemma} --
\label{lem:coexistence}
 Assume that $d \geq 3$.
 Then, $\xi_t$ converges in distribution to an invariant measure in which there is a positive density of both opinions.
\end{lemma}
\begin{proof}
 To prove convergence to a stationary distribution, we first observe that there is no +1 individual in the set $A$ at
 time $t$ if and only if there is no +1 individual in the corresponding dual process at time 0.
 In particular, identifying $\xi_t$ with the set of +1 individuals, we get
\begin{equation}
\label{eq:coexistence}
 P \,(\xi_t \cap A = \varnothing) \ = \ E \,\Big((1 - \theta)^{|\hat \xi_t (A, t)|} \Big).
\end{equation}
 The dominated convergence theorem implies that both terms in \eqref{eq:coexistence} have a limit as time goes to infinity, which proves
 the existence of a stationary distribution.
 Moreover, using again the duality relationship \eqref{eq:duality} and the fact that symmetric simple random walks are transient in three or higher
 dimensions, we obtain the positivity of the limit
 $$ \begin{array}{rcl}
    \lim_{\,t \to \infty} \ P \,(\xi_t (x) \neq \xi_t (y)) & = &
    \lim_{\,t \to \infty} \ P \,(\xi_0 (Z_t (x)) \neq \xi_0 (Z_t (y))) \vspace*{6pt} \\ & = &
    \lim_{\,t \to \infty} \ 2 \theta (1 - \theta) \ P \,(Z_t (x) \neq Z_t (y)) \ > \ 0. \end{array} $$
 This shows that the spatial public debate model converges to a stationary distribution in which the density of +1 individuals and the
 density of $-1$ individuals are both positive.
\end{proof}


\end{document}